\documentclass[11pt,notitlepage,reqno]{amsart}
\usepackage[margin=1in]{geometry}
\usepackage{mathtools} 
\usepackage{amssymb} 
\usepackage{mathrsfs}
\usepackage{bbm}
\usepackage{enumitem}
\usepackage{setspace}
\usepackage[dvipsnames]{xcolor}
\usepackage{multicol}
\usepackage{lastpage}
\usepackage{float}
\usepackage{graphicx}
\usepackage{outlines}
\usepackage[labelfont=small,sc]{caption}
\usepackage{hyperref}
\usepackage[noabbrev,capitalise]{cleveref}
\usepackage{nameref}
\usepackage{subcaption}
\usepackage{bookmark}
\usepackage{marginnote} 
\bookmarksetup{
	numbered,
	open
}

\newcounter{mylabelcounter}

\makeatletter
\newcommand{\labelText}[2]{%
\text{#1}\refstepcounter{mylabelcounter}%
\immediate\write\@auxout{%
  \string\newlabel{#2}{{1}{\thepage}{{\unexpanded{#1}}}{mylabelcounter.\number\value{mylabelcounter}}{}}%
}%
}
\makeatother

\usepackage{amsthm, thmtools}
\newtheorem{theorem}{Theorem}[section]
\newtheorem*{theorem*}{Theorem}
\newtheorem{lemma}[theorem]{Lemma}
\newtheorem{proposition}[theorem]{Proposition}
\newtheorem*{proposition*}{Proposition}
\newtheorem{corollary}[theorem]{Corollary}
\newtheorem*{corollary*}{Corollary}

\newtheorem*{problem*}{Problem}

\theoremstyle{plain} 
\newcommand{\thistheoremname}{}
\newtheorem{genericthm}[theorem]{\thistheoremname}
\makeatletter
\NewDocumentEnvironment{namedthm}{ m o }
{\renewcommand{\thistheoremname}{#1}%
    \begin{genericthm}[#2]
        \def\@currentlabelname{#1}
    }
    {\end{genericthm}}
\makeatother  
\newtheorem*{genericthm*}{\thistheoremname}
\makeatletter
\newenvironment{namedthm*}[1]
{\renewcommand{\thistheoremname}{#1}%
    \begin{genericthm*}%
        \def\@currentlabelname{#1}}%
    {\end{genericthm*}}
\makeatother

\theoremstyle{remark}
\newtheorem{remark}[theorem]{Remark}

\newtheorem{example}[theorem]{Example}

\newtheorem*{claim*}{Claim}

\numberwithin{equation}{section}
\numberwithin{theorem}{section}
\usepackage{chngcntr}

\allowdisplaybreaks
\makeatletter
\def\@tocline#1#2#3#4#5#6#7{\relax
  \ifnum #1>\c@tocdepth 
  \else
    \par \addpenalty\@secpenalty\addvspace{#2}%
    \begingroup \hyphenpenalty\@M
    \@ifempty{#4}{%
      \@tempdima\csname r@tocindent\number#1\endcsname\relax
    }{%
      \@tempdima#4\relax
    }%
    \parindent\z@ \leftskip#3\relax \advance\leftskip\@tempdima\relax
    \rightskip\@pnumwidth plus4em \parfillskip-\@pnumwidth
    #5\leavevmode\hskip-\@tempdima
      \ifcase #1
       \or\or \hskip 1em \or \hskip 2em \else \hskip 3em \fi%
      #6\nobreak\relax
    \dotfill\hbox to\@pnumwidth{\@tocpagenum{#7}}\par
    \nobreak
    \endgroup
  \fi}
\makeatother

\usepackage{tikz}
\usetikzlibrary{matrix,chains,shapes,arrows,scopes,positioning}

\newcommand{\set}[1]{\left\{#1\right\}}

\newcommand{\abs}[1]{\lvert#1\rvert}

\newcommand{\NN}{\mathbb{N}}

\newcommand{\Part}{\mathcal{P}}
\newcommand{\sym}{\mathfrak{S}}
\newcommand{\alt}{\mathfrak{A}}
\newcommand{\Irr}{\operatorname{Irr}}
\newcommand{\IBr}{\operatorname{IBr}}
\newcommand{\Block}{\operatorname{Bl}}

\newcommand{\res}{\big\downarrow}

\begin{document}

\title{Character degrees in $2$-blocks of $\sym_n$ and $\alt_n$}

\author{Bim Gustavsson}
\address[B.~Gustavsson]{School of Mathematics, Watson Building, University of Birmingham, Edgbaston, Birmingham B15 2TT, UK}
\email{bxg281@bham.ac.uk}

\begin{abstract}
	Let $p$ be an odd prime. We show that for sufficiently large $n$, every $2$-block of $\sym_n$ and $\alt_n$ contains an ordinary irreducible character of degree divisible by $p$. For almost all $2$-blocks of $\alt_n$, we classify whether it contains a rational valued ordinary irreducible character of degree divisible by $p$.
\end{abstract}

\maketitle

\section{Introduction}
	A recent milestone in the representation theory of finite groups is the resolution of Brauer's Height Zero Conjecture, whose proof was completed in \cite{MNST24}. A corollary of this is that we can determine whether a $q$-block $B$ of a finite group $G$ contains an irreducible character $\chi$ whose degree is divisible by $q$. A generalisation of this corollary is the following: For primes $p\neq q$, when does a $q$-block of a finite group $G$ contain an irreducible character whose degree is divisible by $p$? In \cite{GMV19} the authors answer this question when $B$ is the principal block of a finite group $G$. For $q=2$ and $p$ odd, \cite[Theorem A]{GMV19} describes when the principal $q$-block of $G$ contains an irreducible character whose degree is even. If $G$ is the symmetric group $\sym_n$ or the alternating group $\alt_n$, then \cite[Theorem C]{GMV19} states that if $n\geq 5$ and $p<q\leq n$ are primes, then the principal $q$-block of $G$ contains an irreducible character whose degree is divisible by $p$. We will now focus our attention on the symmetric and alternating groups. Let $G\in\set{\sym_n,\alt_n}$ for some integer $n\geq q$. If $n\geq 5$ and $p<q\leq n$, then \cite[Theorem C]{GMV19} tells us that the principal $q$-block of $G$ contains an irreducible character whose degree is divisible by $p$. In turn, \cite[Theorem A]{GMei21} generalises \cite[Theorem C]{GMV19} by removing the condition that $p<q$. Furthermore, \cite[Theorem B]{GMei21} classify when the principal block of $\alt_n$ contains an irreducible and rational valued $p$-divisible character. The case of irreducible characters in non-principal $q$-blocks whose degree is divisible by $p$ was investigated in \cite{GMec24} for $p=2$ and odd primes $q$. In this article we continue this line of investigation, by resolving the question in the case of non-principal $q$-blocks of $\sym_n$ and $\alt_n$ for $q=2$ and $p$ any odd prime. 
	
	Let $q$ be a prime. For a finite group $G$, we denote by $\Irr(G)$ the complete set of irreducible complex characters of $G$ and let $\IBr(G)$ denote the irreducible Brauer characters of $G$ in characteristic $q$. The set $\Irr(G)\cup\IBr(G)$ can be divided up into certain equivalence classes, called \emph{$q$-blocks}, for details see \cite{N98}. We denote the set of all $q$-blocks of $G$ by $\Block_q(G)$ and for $B\in \Block_q(G)$ let $\Irr(B)=B\cap\Irr(G)$. The $q$-block that contains the trivial character is called the \emph{principal $q$-block}. For $t\in\NN$, we say that a character $\chi$ of $G$ is \emph{$t$-divisible} if $t\mid \chi(1)$. Before we state our main theorem, we need to briefly introduce blocks of $\sym_n$ and $\alt_n$. For a detailed description of the $2$-blocks of $\sym_n$ and of $\alt_n$, see \cref{sec:prelims} and \cref{sec:deg in alt} respectively. Let $B_c$ denote the $2$-block of $\sym_n$ that is labelled by the unique $2$-core partition $\gamma_c$ with $c$ parts, that is $\gamma_c = (c,c-1,\dots,1)$ if $c\in \NN$ and $\gamma_0=\varnothing$. For example, $B_0$ and $B_1$ are the principal $2$-blocks of $\sym_n$ when $n$ is even or odd respectively. Our main theorem is as follows:
	
	\begin{theorem}\label{thm: char in block}
		Let $p$ be an odd prime, $n\geq p$ be a natural number and let $r,a,k\in\NN_0$ be such that $n=ap^k+r$, $a < p$ and $r<p^k$. Let $G\in\set{\sym_n,\alt_n}$ and let $B\in\Block_2(G)$ be a $2$-block of $G$. If $G=\sym_n$ then suppose that $B=B_c$ and if $G=\alt_n$ then suppose that $B$ is a $2$-block of $\alt_n$ covered by the $2$-block $B_c$ of $\sym_n$. Suppose that either 
		\begin{enumerate}
			\item $2 \leq c\leq 10$ and $n\geq 66$, or \label{1}
			\item $c\geq 11$ and $p^k\geq 2c-1$, \label{2}
		\end{enumerate}
		then there exists some $p$-divisible $\chi\in \Irr(B)$. Moreover, if $G = \alt_n$ and either (\ref{1}) or (\ref{2}) holds, then $\Irr(B)$ contains a rational valued $p$-divisible character if and only if $w(B)>0$.
	\end{theorem}	

	\begin{corollary}\label{cor: N}
		Let $p$ be an odd prime, $n\geq p$ be a natural number and let $G\in\set{\sym_n,\alt_n}$. If $n\geq 8p^2+2p-4$, then every $2$-block of $G$ contains a $p$-divisible irreducible character.
	\end{corollary}

	The structure of this article is as follows: In \cref{sec:prelims} we recall necessary facts from the representation theory of symmetric groups and related combinatorics. In \cref{sec:deg in sym} we prove \cref{thm: char in block} and \cref{cor: N} for $G=\sym_n$. We also make an important observation about the proof of \cref{thm: main part A} needed for the proof of \cref{thm: main part A} when $G=\alt_n$. We begin \cref{sec:deg in alt} by introducing necessary facts about the representation theory of alternating groups and its block decomposition. We then proceed to prove \cref{thm: char in block} and \cref{cor: N} for $G=\alt_n$.

\bigskip

\section{Preliminaries}\label{sec:prelims}
	For $n\in \NN$. A \emph{partition} $\lambda$ of $n$ is a weakly decreasing sequence of non-negative integers $(\lambda_1,\lambda_2, \dots ,\lambda_t)$ such that $\sum_{i=1}^t \lambda_i = n$, and we denote the set of all partitions of $n$ by $\Part(n)$. We call $n$ the \emph{size} of $\lambda$, which we denote by $\abs{\lambda}=n$ and we call $\lambda_1,\dots,\lambda_t$ the \emph{parts} of $\lambda$. The number of non-zero parts of $\lambda$ is called the \emph{length} of $\lambda$ and is denoted by $\ell(\lambda)$. The \emph{conjugate} of $\lambda$ is the partition $\lambda':= (\mu_1, \mu_2, \dots, \mu_{\ell(\lambda')})$ where $\mu_i = \abs{\set{j\in \NN \mid \lambda_j \geq i}}$. We say that $\lambda$ is \emph{self-conjugate} if $\lambda=\lambda'$. For $s,t\in\NN_0$, let $[s,t] := \set{s,s+1,\dots,t}$ if $s\leq t$ and $[s,t]:=\varnothing$ otherwise.
	
	\subsection{Representation theory of the symmetric groups}
		For further background on the content of this section, we refer the reader to \cite{O94} and \cite{N98}. Let $n\in \NN$. We denote by $\sym_n$ the symmetric group on $n$ objects. It is well known that $\Irr(\sym_n)$ is in natural bijection with $\Part(n)$ and we denote by $\chi^\lambda$ the irreducible character of $\sym_n$ corresponding to the partition $\lambda\in\Part(n)$. The \emph{Young diagram} of $\lambda$ is defined to be the set
			\[ Y(\lambda) := \set{(i,j) \in \NN \times \NN \mid 1\leq i \leq \ell(\lambda),\ 1 \leq j \leq \lambda_i}. \]
		For each $(i,j)\in Y(\lambda)$ we can associate a \emph{hook} of $\lambda$ at $(i,j)\in Y(\lambda)$ which is defined as:
			\[ H_{i,j}(\lambda) := \set{(i,j')\in Y(\lambda) \mid j' \geq j} \cup \set{(i',j) \in Y(\lambda) \mid i' \geq i} \]
		and we let $h_{i,j}(\lambda) = \abs{H_{i,j}(\lambda)}$. We say that $H$ is a \emph{hook of $\lambda$} if $H=H_{i,j}(\lambda)$ for some $(i,j)\in Y(\lambda)$. For $(i,j)\in Y(\lambda)$, the \emph{rim hook of $H_{i,j}(\lambda)$} is the set
			\[ R_{i,j}(\lambda) := \set{ (x+t,y+t)\in Y(\lambda) \mid (x,y)\in H_{i,j}(\lambda)\ \text{and}\ (x+t+1,y+t+1)\not\in Y(\lambda)}.\]
		Let $e\in \NN$. If $e = h_{i,j}(\lambda)$ for some $(i,j)\in Y(\lambda)$, then we say that $\lambda$ has an $e$-hook. We denote by $\lambda\setminus H_{i,j}(\lambda)$ the partition of size $n-e$ which satisfies
			\[ Y(\lambda\setminus H_{i,j}(\lambda)) = Y(\lambda)\setminus R_{i,j}(\lambda). \]
		
		A partition is called an \emph{$e$-core} if it has no $e$-hooks. If $\lambda$ is not an $e$-core, we can obtain the $e$-core of $\lambda$ by successively removing $e$-hooks until no more $e$-hooks can be removed. It is well known that the order in which we remove $e$-hooks from $\lambda$ does not change the resulting $e$-core, see for example \cite[Theorem 2.7.16]{JK81}. We denote the $e$-core of $\lambda$ by $C_e(\lambda)$. An example of these concepts are illustrated in \cref{fig: partition hook removal}.

		\begin{figure}[H]
		\centering
		\captionsetup{justification=centering}
			\begin{subfigure}[t]{0.15\textwidth}
    	\begin{tikzpicture}[scale=0.4,every node/.style={scale=.8}]
        \filldraw[draw = black, fill = orange] (3,0) -- (5,0) -- (5,-1) -- (4,-1) -- (4,-2) -- (2,-2) -- (2,-1) -- (3,-1) -- (3,0) ;
        \draw (0,0) -- (5,0);
        \draw (0,-1) -- (5,-1);
        \draw (0,-2) -- (4,-2);
        \draw (0,-3) -- (2,-3);
        \draw (0,-4) -- (2,-4);
        \draw (0,-5) -- (1,-5); 
        \draw (0,0) -- (0,-5);
        \draw (1,0) -- (1,-5);
        \draw (2,0) -- (2,-4);
        \draw (3,0) -- (3,-2);
        \draw (4,0) -- (4,-2);
        \draw (5,0) -- (5,-1);
				\draw (2.5,-0.5) node[] {$4$};
      \end{tikzpicture}
			\caption{}
			\end{subfigure}
			~
			\begin{subfigure}[t]{0.15\textwidth}
     \begin{tikzpicture}[scale=0.4,every node/.style={scale=.8}]
        \draw [dotted] (3,0) -- (5,0) -- (5,-1) -- (4,-1) -- (4,-2) -- (2,-2) -- (2,-1) -- (3,-1) -- (3,0) ;
        \draw (0,0) -- (3,0);
        \draw (0,-1) -- (3,-1);
        \draw (0,-2) -- (2,-2);
        \draw (0,-3) -- (2,-3);
        \draw (0,-4) -- (2,-4);
        \draw (0,-5) -- (1,-5); 
        \draw (0,0) -- (0,-5);
        \draw (1,0) -- (1,-5);
        \draw (2,0) -- (2,-4);
        \draw (3,0) -- (3,-1);
      \end{tikzpicture}
			\caption{}
		\end{subfigure}
		\begin{subfigure}[t]{0.15\textwidth}
			\begin{tikzpicture}[scale=0.4,every node/.style={scale=.8}]
        \draw [dotted] (3,0) -- (5,0) -- (5,-1) -- (4,-1) -- (4,-2) -- (2,-2) -- (2,-1) -- (3,-1) -- (3,0) ;
				\filldraw[draw = black, fill = SeaGreen] (1,-2) -- (2,-2) -- (2,-4) -- (1,-4) -- (1,-5) -- (0,-5) -- (0,-3) -- (1,-3);
       \draw (0,0) -- (3,0);
        \draw (0,-1) -- (3,-1);
        \draw (0,-2) -- (2,-2);
        \draw (0,-3) -- (2,-3);
        \draw (0,-4) -- (2,-4);
        \draw (0,-5) -- (1,-5); 
        \draw (0,0) -- (0,-5);
        \draw (1,0) -- (1,-5);
        \draw (2,0) -- (2,-4);
        \draw (3,0) -- (3,-1);
				\draw (0.5,-2.5) node[] {$4$};
      \end{tikzpicture}
			\caption{}
			\end{subfigure}
			~
			\begin{subfigure}[t]{0.15\textwidth}
      \begin{tikzpicture}[scale=0.4,every node/.style={scale=.8}]
        \draw [dotted] (3,0) -- (5,0) -- (5,-1) -- (4,-1) -- (4,-2) -- (2,-2) -- (2,-1) -- (3,-1) -- (3,0);
				\draw [dotted] (0,-3) -- (0,-5) -- (1,-5) -- (1,-4) -- (2,-4) -- (2,-2);
        \draw (0,0) -- (3,0);
        \draw (0,-1) -- (3,-1);
        \draw (0,-2) -- (2,-2);
        \draw (0,-3) -- (1,-3);
        \draw (0,0) -- (0,-3);
        \draw (1,0) -- (1,-3);
        \draw (2,0) -- (2,-2);
        \draw (3,0) -- (3,-1);
			\end{tikzpicture}
			\caption{}
		\end{subfigure}
		\caption{\\
			(A) Partition $\lambda = (5,4,2^2,1)$ with rim hook $R_{1,3}(\lambda)$ highlighted.  \\
			(B) Removal of $R_{1,3}(\lambda)$ from $\lambda$, resulting in partition $\mu:= (3,2^3,1)$.  \\
			(C) Partition $\mu=(3,2^3,1)$ with rim hook $R_{3,1}(\mu)$ highlighted. \\
			(D) Removal of $R_{3,1}(\mu)$ from $\mu$ resulting in partition $\nu:= (3,2,1)$. In particular, we have that $C_4(\lambda)=\nu$.} 
			\label{fig: partition hook removal}
		\end{figure}
		This hook removal procedure will play an important role in this paper, and so we will now introduce some necessary theory with respect to this.
		
	\subsection{$\beta$-sets and cores}
		A \emph{$\beta$-set} $X$ is a non-empty finite subset of $\NN_0$. We fix the convention that $x_1 > x_2 > \dots > x_t$ whenever we write $X = \set{x_1,\dots,x_t}$. Given a $\beta$-set $X$ we can associate to it a partition defined by 
			\[ P(X):=(x_1 - (t-1),\, x_{2} - (t-2),\,\dots,\,x_t).\] 
	 	From the definition of $P(X)$ is clear that $\abs{P(X)} = x_1+\dots+x_t - \frac{t(t-1)}{2}$. Note that there is not a bijection between $\beta$-sets and partitions. For example, if $X=\set{9,7,6,3,1}$ and $Y=\set{10,8,5,4,2,0}$ then $P(X)=P(Y)=(5,4,2^2,1)$. Given a partition $\lambda$, we say that \emph{$X$ is a $\beta$-set of $\lambda$} if $P(X)=\lambda$.
		
		\begin{proposition}\cite[Corollary 1.5, Proposition 1.8]{O94}\label{prop: remove hook}
			Let $\lambda$ be a partition, $X$ a $\beta$-set of $\lambda$ and let $h\in \NN$. Then $\lambda$ contains a hook $H$ of size $h$ if and only if there exists $x,y\in \NN_0$ such that $x\in X$, $y\not\in X$ and $h=x-y$. Whenever this holds, $(X\setminus\set{x})\cup\set{y}$ is a $\beta$-set for $\lambda\setminus H$.
		\end{proposition}

		Let $e\in \NN$. For a partition $\lambda$, it is not clear whether removing $e$-hooks from $\lambda$ in a different order gives the same $e$-core. However, this becomes easy to see if we represent it on \emph{James' abacus}. To be more precise, \emph{James' $e$-abacus} consists of $e$ vertical runners (columns) labelled $0,\, 1,\, \dots,\, e-1$ from left to right, containing infinitely many rows labelled by $\NN_0$ from top to bottom. Let $x\in[0,e-1]$ and $y\in\NN_0$, then $(x,y)$ labels the position in column $x$ and row $y$. Now, the configuration of $X$ on the $e$-abacus is as follows: for each $b\in X$ place a bead at the unique position $(x,y)$ on the $e$-abacus which satisfies $b=x+ey$. In this setting, manipulations of the beads on the $e$-abacus will correspond to changes to the corresponding partition. For example, if $b\in X$, $b-e\not\in B$ and $b-e\geq 0$, then we know that replacing $b$ with $b-e$ in $X$ corresponds to removing an $e$-hook from the partition $P(X)$ by \cref{prop: remove hook}. On the $e$-abacus, this is equivalent to moving bead $b$ one step up on its runner. Hence, moving all the beads up in an $e$-abacus configuration of a $\beta$-set $X$ as high up as possible on its runner gives a new $\beta$-set $Y$ such that $P(Y)$ is the $e$-core of the partition $P(X)$. For details see \cite{O94}. 
		
		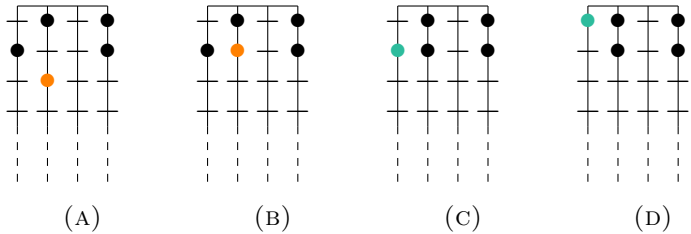
\begin{figure}[H]
  	\begin{subfigure}[t]{.15\textwidth}
			\begin{tikzpicture}[scale=0.4,every node/.style={scale=1.2}]
				\draw (0,0) -- (3,0);
        \draw (0,0) -- (0,-4);
				\draw[dashed] (0,-4) -- (0,-6);
        \draw (1,0) -- (1,-4);
				\draw[dashed] (1,-4) -- (1,-6);
        \draw (2,0) -- (2,-4);
				\draw[dashed] (2,-4) -- (2,-6);
        \draw (3,0) -- (3,-4);
				\draw[dashed] (3,-4) -- (3,-6);
				\draw (0,-0.5) node[] {$-$};
				\draw (1,-0.5) node[] {$\bullet$};
				\draw (2,-0.5) node[] {$-$};
				\draw (3,-0.5) node[] {$\bullet$};
				\draw (0,-1.5) node[] {$\bullet$};
				\draw (1,-1.5) node[] {$-$};
				\draw (2,-1.5) node[] {$-$};
				\draw (3,-1.5) node[] {$\bullet$};
				\draw (0,-2.5) node[] {$-$};
				\draw (1,-2.5) node[orange] {$\bullet$};
				\draw (2,-2.5) node[] {$-$};
				\draw (3,-2.5) node[] {$-$};
				\draw (0,-3.5) node[] {$-$};
				\draw (1,-3.5) node[] {$-$};
				\draw (2,-3.5) node[] {$-$};
				\draw (3,-3.5) node[] {$-$};
      \end{tikzpicture}
			\caption{}
		\end{subfigure}
		~
		\begin{subfigure}[t]{.15\textwidth}
    	\begin{tikzpicture}[scale=0.4,every node/.style={scale=1.2}]
				\draw (0,0) -- (3,0);
        \draw (0,0) -- (0,-4);
				\draw[dashed] (0,-4) -- (0,-6);
        \draw (1,0) -- (1,-4);
				\draw[dashed] (1,-4) -- (1,-6);
        \draw (2,0) -- (2,-4);
				\draw[dashed] (2,-4) -- (2,-6);
        \draw (3,0) -- (3,-4);
				\draw[dashed] (3,-4) -- (3,-6);
				\draw (0,-0.5) node[] {$-$};
				\draw (1,-0.5) node[] {$\bullet$};
				\draw (2,-0.5) node[] {$-$};
				\draw (3,-0.5) node[] {$\bullet$};
				\draw (0,-1.5) node[] {$\bullet$};
				\draw (1,-1.5) node[orange] {$\bullet$};
				\draw (2,-1.5) node[] {$-$};
				\draw (3,-1.5) node[] {$\bullet$};
				\draw (0,-2.5) node[] {$-$};
				\draw (1,-2.5) node[] {$-$};
				\draw (2,-2.5) node[] {$-$};
				\draw (3,-2.5) node[] {$-$};
				\draw (0,-3.5) node[] {$-$};
				\draw (1,-3.5) node[] {$-$};
				\draw (2,-3.5) node[] {$-$};
				\draw (3,-3.5) node[] {$-$};
      \end{tikzpicture}
			\caption{}
		\end{subfigure}
		~
		\begin{subfigure}[t]{.15\textwidth}
      \begin{tikzpicture}[scale=0.4,every node/.style={scale=1.2}]
				\draw (0,0) -- (3,0);
        \draw (0,0) -- (0,-4);
				\draw[dashed] (0,-4) -- (0,-6);
        \draw (1,0) -- (1,-4);
				\draw[dashed] (1,-4) -- (1,-6);
        \draw (2,0) -- (2,-4);
				\draw[dashed] (2,-4) -- (2,-6);
        \draw (3,0) -- (3,-4);
				\draw[dashed] (3,-4) -- (3,-6);
				\draw (0,-0.5) node[] {$-$};
				\draw (1,-0.5) node[] {$\bullet$};
				\draw (2,-0.5) node[] {$-$};
				\draw (3,-0.5) node[] {$\bullet$};
				\draw (0,-1.5) node[SeaGreen] {$\bullet$};
				\draw (1,-1.5) node[] {$\bullet$};
				\draw (2,-1.5) node[] {$-$};
				\draw (3,-1.5) node[] {$\bullet$};
				\draw (0,-2.5) node[] {$-$};
				\draw (1,-2.5) node[] {$-$};
				\draw (2,-2.5) node[] {$-$};
				\draw (3,-2.5) node[] {$-$};
				\draw (0,-3.5) node[] {$-$};
				\draw (1,-3.5) node[] {$-$};
				\draw (2,-3.5) node[] {$-$};
				\draw (3,-3.5) node[] {$-$};
      \end{tikzpicture}
			\caption{} 
		\end{subfigure} 
		~
		\begin{subfigure}[t]{.15\textwidth}
      \begin{tikzpicture}[scale=0.4,every node/.style={scale=1.2}]
				\draw (0,0) -- (3,0);
        \draw (0,0) -- (0,-4);
				\draw[dashed] (0,-4) -- (0,-6);
        \draw (1,0) -- (1,-4);
				\draw[dashed] (1,-4) -- (1,-6);
        \draw (2,0) -- (2,-4);
				\draw[dashed] (2,-4) -- (2,-6);
        \draw (3,0) -- (3,-4);
				\draw[dashed] (3,-4) -- (3,-6);
				\draw (0,-0.5) node[SeaGreen] {$\bullet$};
				\draw (1,-0.5) node[] {$\bullet$};
				\draw (2,-0.5) node[] {$-$};
				\draw (3,-0.5) node[] {$\bullet$};
				\draw (0,-1.5) node[] {$-$};
				\draw (1,-1.5) node[] {$\bullet$};
				\draw (2,-1.5) node[] {$-$};
				\draw (3,-1.5) node[] {$\bullet$};
				\draw (0,-2.5) node[] {$-$};
				\draw (1,-2.5) node[] {$-$};
				\draw (2,-2.5) node[] {$-$};
				\draw (3,-2.5) node[] {$-$};
				\draw (0,-3.5) node[] {$-$};
				\draw (1,-3.5) node[] {$-$};
				\draw (2,-3.5) node[] {$-$};
				\draw (3,-3.5) node[] {$-$};
      \end{tikzpicture}
			\caption{}
		\end{subfigure} 
		\caption{The partitions in \cref{fig: partition hook removal} as seen on a $4$-abacus with $\beta$-sets of size 5. The highlighted bead $b$ in (A) (and (C) respectively) corresponds to the highlighted rim hooks in \cref{fig: partition hook removal} (A) (and (C) respectively). Moving $b$ up one step on its runner yields the abacus (B) (and (D) respectively), which corresponds to removing the highlighted rim hooks in \cref{fig: partition hook removal} (A) (and (C) respectively).} 
		\label{fig: abacus}
		\end{figure}

		Let us now briefly explain the importance of computing cores of partitions. Let $q\leq n$ be a prime. Nakayama's conjecture (which was proved in 1947 by Brauer and Robinson, see \cite{B47} and \cite{R47}) states that $\chi^\lambda,\chi^\mu\in \Irr(\sym_n)$ belong to the same $q$-block if and only if $C_q(\lambda)=C_q(\mu)$, i.e. $q$-blocks of $\sym_n$ are parametrised by $q$-cores. Let $\lambda\in\Part(n)$. The number of $q$-hooks that needs to be removed from $\lambda$ to get to the $q$-core of $\lambda$ is called the \emph{$q$-weight of $\lambda$} and is denoted by $w_q(\lambda)$, or simply $w(\lambda)$ if it is clear from context. Furthermore, suppose $B\in\Block_q(\sym_n)$ is parametrised by the $q$-core partition $\gamma$, then $w(\lambda)=\frac{n-\abs{\gamma}}{q}$ for all $\chi^\lambda\in\Irr(B)$, which we call the \emph{weight} of $B$ and denote it by $w(B)$. We define the partition $\gamma_c := (c,c-1,\dots,1)$ for $c\in \NN$ and $\gamma_0 := \varnothing$. Hence, any $2$-block of $\sym_n$ is parameterised by $\gamma_c$ for some $c\in\NN_0$. 

		For $X$ a $\beta$-set, let $X_0 = \set{x \in X \mid x\ \text{even}}$ and $X_1 = X \setminus X_0$. The following lemma is immediate by considering the configuration of a $\beta$-set $X$ on a $2$-abacus.

		\begin{lemma}\label{lem: in core c}
		Let $X$ be a $\beta$-set. Then $C_2(P(X))=\gamma_c$, where
			\[ c = \begin{cases}
				\abs{X_1} - \abs{X_0} &\text{if}\ \abs{X_1} \geq \abs{X_0},\\
				\abs{X_0} - \abs{X_1}-1 &\text{otherwise}.
			\end{cases}\]
		\end{lemma}
		
		We have the following relation between $2$-blocks and $\beta$-sets:

		\begin{lemma}\label{lem: even is even}
			Let $X$ be a $\beta$-set, $n=\abs{P(X)}$ and let $t=\abs{X}$. If $B_{t}\in\Block_2(\sym_n)$, then $\abs{X_0}$ is even.
		\end{lemma} 
		
		\begin{proof}
			Since $B_t\in \Block_2(\sym_n)$ we have that $\frac{t(t+1)}{2}\equiv n\ (\text{mod}\ 2)$ and so $n+\frac{t(t+1)}{2}\equiv 0\ (\text{mod}\ 2)$. Furthermore, we have that
				\begin{align}\label{eq: even odd}
					 \sum_{x\in X} x = n + \frac{t(t-1)}{2} = n+\frac{t(t+1)}{2}-t.
				\end{align}
			Now suppose that $t$ is even. Then $\sum_{x\in X} x \equiv 0\ (\text{mod}\ 2)$ by \eqref{eq: even odd} which implies that $\abs{X_1}$ is even. On the other hand, suppose that $t$ is odd. Then $\sum_{x\in X} x \equiv 1\ (\text{mod}\ 2)$ by \eqref{eq: even odd} which in turn implies that $\abs{X_1}$ is odd. Since $t=\abs{X_0}=\abs{X}-\abs{X_1}$ we have that $\abs{X_0}$ is even.
		\end{proof}

		So far we have used $2$-cores to understand the distribution of $\Irr(\sym_n)$ in $2$-blocks of $\sym_n$. We will now consider another useful aspect of cores. The $p$-valuation of the degree of ordinary irreducible characters was described in \cite{M71} in terms of $p$-cores and $p$-quotients. The following proposition is a consequence of this work, and will play a key role in the proof of \cref{thm: char in block}.

		\begin{proposition}\label{prop: core doesnt fit}
			Let $p$ be a prime, $n\geq p$ a natural number and let $\chi^\lambda\in \Irr(\sym_n)$. Let $r,a,k\in\NN_0$ be such that $n=ap^k+r$, $a<p$ and $r<p^k$. Then 
				\[ \chi^\lambda\ \text{is}\ p\text{-divisible if and only if}\ \abs{C_{p^k}(\lambda)} > r\ \text{or}\ \chi^{C_{p^k}(\lambda)}\ \text{is}\ p\text{-divisible}. \]
		\end{proposition}
\bigskip

\section{Character degrees in $2$-blocks of $\sym_n$}\label{sec:deg in sym}
	In this section we prove \cref{thm: char in block} for $\sym_n$.  The proof for $\alt_n$ is postponed until \cref{sec:deg in alt}. For $t\in \NN_0$, let $\sigma_t\in\set{1,2}$ be such that $\sigma_t \equiv t-1\ (\text{mod}\ 2)$ and let $\delta_{t,1}$ denote the Kroenecker delta, which takes value $1$ if $t=1$ and $0$ otherwise.

	\begin{theorem}\label{thm: main part A}
		Let $p$ be an odd prime, $n\geq p$ be a natural number and let $r,a,k\in\NN_0$ be such that $n=ap^k+r$, $a<p$ and $r<p^k$. Let $B_c\in \Block_2(\sym_n)$ be a $2$-block of $\sym_n$. If either
		\begin{enumerate}
			\item $2 \leq c\leq 10$ and $n\geq 66$, or
			\item $c\geq 11$ and $p^k\geq 2c-1$,
		\end{enumerate}
		then there exists some $p$-divisible $\chi\in\Irr(B_c)$. 
	\end{theorem}

	Before we proceed with the proof, we give an example that illustrates some of the key ideas behind the proof of \cref{thm: main part A}. 
	\begin{example}
		Let $n=75$, $p=11$ and $c=5$. Let $r,a,k\in \NN_0$ be such that $n=ap^k+r$, $a<p$ and $r<p^k$, so in this example $r=9$, $a=6$ and $k=1$. Recall that if $X=\set{x_1,x_2,\dots,x_t}$ is a $\beta$-set, then the associated partition of $X$ is defined as $P(X)=(x_1-(t-1),x_2-(t-2),\dots,x_t)$ and so $\abs{P(X)} = x_1 + x_2 + \dots x_t - \frac{t(t-1)}{2}$.
		
		We will construct a $\beta$-set $X$ such that $C_2(P(X))=\gamma_5$, $P(X)=75$ and $\abs{C_{11}(P(X))}\neq 9$. This implies that $\chi^{P(X)}\in\Irr(B_5)$ is $11$-divisible by \cref{prop: core doesnt fit}, where $B_5$ is the $2$-block of $\sym_{75}$ indexed by the $2$-core partition $\gamma_5=(5,4,3,2,1)$. In order to construct such a $\beta$-set, we start with a $\beta$-set $Y$ such that $\abs{Y}$ is minimal and $C_2(P(Y))=\gamma_5$, which is $Y=\set{9,7,5,3,1}$. The $2$-abacus configuration of $Y$ is illustrated in \cref{fig: proof example} (A).

		We will now proceed to move beads down along the runners of the $2$-abacus configuration of $Y$, by moving one bead at a time. The bead in the largest position, namely that in position 9, is moved down along its runner until \emph{either} the corresponding partition is of size $n$ \emph{or} the bead is in position $ap^k-\sigma_a$, i.e. the row immediately above $ap^k$. In this example $ap^k-\sigma_a = 65$ and the corresponding partition is of size $65+7+5+3+1-10 = 71 < 75$. This gives us the $2$-abacus configuration as shown in \cref{fig: proof example} (B). Since this partition is of size less than $n$, we will continue moving beads down. 
		
		Next we move the bead in the second largest position down along its runner. We do so until \emph{either} the corresponding partition is of size $n$ \emph{or} the bead is in the position $p^k-2$. In this example the second largest bead moves from position $7$ to position $9$ and we get the $2$-abacus configuration as shown in \cref{fig: proof example} (C). The corresponding partition is of size $73$. Since $73<75$ we will continue moving beads down.
		
		For the remaining beads on the right-hand runner of the $2$-abacus, we move them down one bead at a time, not allowing beads to pass each other on the abacus. We move a bead down until \emph{either} the corresponding partition is of size $n$ \emph{or} the bead can no longer move down the runner. Note that moving the bead at the third largest position down one step on its runner gives us the $2$-abacus configuration as illustrated in \cref{fig: proof example} (D). The $\beta$-set of this $2$-abacus configuration is $X=\set{65,9,7,3,1}$ and note that $\abs{P(X)}= 65+9+7+3+1-10 = 75$. Furthermore, all $5$ beads are on the second runner and so $C_2(P(X))=\gamma_5$. Hence $\chi^{P(X)}\in \Irr(B_c)$ for $B_c\in \Block_2(\sym_n)$.
		
		Lastly, note that $65<6\cdot 11$ and $1,3,7,9 < 11$, which implies that $\abs{C_{11}(P(X))}>9$. It follows that $\chi^{P(X)}$ is $11$-divisible by \cref{prop: core doesnt fit}.

		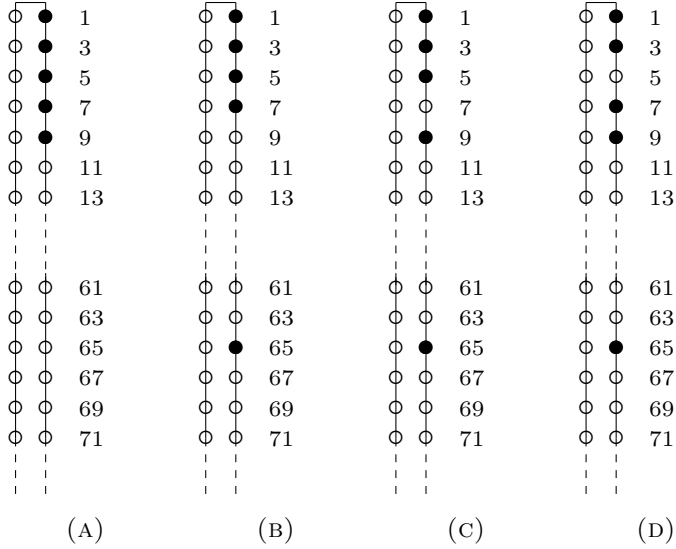
\begin{figure}[H]
			\captionsetup{justification=centering}
			\begin{subfigure}[t]{.15\textwidth}
			\begin{tikzpicture}[scale=0.4,every node/.style={scale=1.2}]
				\draw (0,-.5) -- (1,-.5);
        \draw (0,-.5) -- (0,-7);
				\draw[dashed] (0,-7) -- (0,-10);
				\draw (0,-9.5) -- (0,-15);
				\draw[dashed] (0,-15) -- (0,-17);
        \draw (1,-.5) -- (1,-7);
				\draw[dashed] (1,-7) -- (1,-10);
				\draw (1,-9.5) -- (1,-15);
				\draw[dashed] (1,-15) -- (1,-17);
				\foreach \n in {1,...,7}{\draw (0,-\n) node[] {$\circ$};}
				\foreach \n in {10,...,15}{\draw (0,-\n) node[] {$\circ$};}
				\foreach \n in {1,...,7}{\pgfmathsetmacro{\sumval}{2*\n-1}
    		\draw (1,-\n) node[label=right:{\tiny\pgfmathprintnumber[fixed,precision=1]{\sumval}}] {$\circ$};}
				\foreach \n in {10,...,15}{\pgfmathsetmacro{\sumval}{2*\n+41}
    		\draw (1,-\n) node[label=right:{\tiny\pgfmathprintnumber[fixed,precision=1]{\sumval}}] {$\circ$};}
				\foreach \n in {1,...,5}{\draw (1,-\n) node[] {$\bullet$};}
      \end{tikzpicture}
			\caption{}
		\end{subfigure}
		~
		\begin{subfigure}[t]{.15\textwidth}
			\begin{tikzpicture}[scale=0.4,every node/.style={scale=1.2}]
				\draw (0,-.5) -- (1,-.5);
        \draw (0,-.5) -- (0,-7);
				\draw[dashed] (0,-7) -- (0,-10);
				\draw (0,-9.5) -- (0,-15);
				\draw[dashed] (0,-15) -- (0,-17);
        \draw (1,-.5) -- (1,-7);
				\draw[dashed] (1,-7) -- (1,-10);
				\draw (1,-9.5) -- (1,-15);
				\draw[dashed] (1,-15) -- (1,-17);
				\foreach \n in {1,...,7}{\draw (0,-\n) node[] {$\circ$};}
				\foreach \n in {10,...,15}{\draw (0,-\n) node[] {$\circ$};}
				\foreach \n in {1,...,7}{\pgfmathsetmacro{\sumval}{2*\n-1}
    		\draw (1,-\n) node[label=right:{\tiny\pgfmathprintnumber[fixed,precision=1]{\sumval}}] {$\circ$};}
				\foreach \n in {10,...,15}{\pgfmathsetmacro{\sumval}{2*\n+41}
    		\draw (1,-\n) node[label=right:{\tiny\pgfmathprintnumber[fixed,precision=1]{\sumval}}] {$\circ$};}
				\foreach \n in {1,...,4}{\draw (1,-\n) node[] {$\bullet$};}
				\draw (1,-12) node[] {$\bullet$};
      \end{tikzpicture}
			\caption{}
		\end{subfigure}
		~
		\begin{subfigure}[t]{.15\textwidth}
			\begin{tikzpicture}[scale=0.4,every node/.style={scale=1.2}]
				\draw (0,-.5) -- (1,-.5);
        \draw (0,-.5) -- (0,-7);
				\draw[dashed] (0,-7) -- (0,-10);
				\draw (0,-9.5) -- (0,-15);
				\draw[dashed] (0,-15) -- (0,-17);
        \draw (1,-.5) -- (1,-7);
				\draw[dashed] (1,-7) -- (1,-10);
				\draw (1,-9.5) -- (1,-15);
				\draw[dashed] (1,-15) -- (1,-17);
				\foreach \n in {1,...,7}{\draw (0,-\n) node[] {$\circ$};}
				\foreach \n in {10,...,15}{\draw (0,-\n) node[] {$\circ$};}
				\foreach \n in {1,...,7}{\pgfmathsetmacro{\sumval}{2*\n-1}
    		\draw (1,-\n) node[label=right:{\tiny\pgfmathprintnumber[fixed,precision=1]{\sumval}}] {$\circ$};}
				\foreach \n in {10,...,15}{\pgfmathsetmacro{\sumval}{2*\n+41}
    		\draw (1,-\n) node[label=right:{\tiny\pgfmathprintnumber[fixed,precision=1]{\sumval}}] {$\circ$};}
				\foreach \n in {1,...,3}{\draw (1,-\n) node[] {$\bullet$};}
				\draw (1,-12) node[] {$\bullet$};
				\draw (1,-5) node[] {$\bullet$};
      \end{tikzpicture}
			\caption{}
		\end{subfigure}
		~
		\begin{subfigure}[t]{.15\textwidth}
			\begin{tikzpicture}[scale=0.4,every node/.style={scale=1.2}]
				\draw (0,-.5) -- (1,-.5);
        \draw (0,-.5) -- (0,-7);
				\draw[dashed] (0,-7) -- (0,-10);
				\draw (0,-9.5) -- (0,-15);
				\draw[dashed] (0,-15) -- (0,-17);
        \draw (1,-.5) -- (1,-7);
				\draw[dashed] (1,-7) -- (1,-10);
				\draw (1,-9.5) -- (1,-15);
				\draw[dashed] (1,-15) -- (1,-17);
				\foreach \n in {1,...,7}{\draw (0,-\n) node[] {$\circ$};}
				\foreach \n in {10,...,15}{\draw (0,-\n) node[] {$\circ$};}
				\foreach \n in {1,...,7}{\pgfmathsetmacro{\sumval}{2*\n-1}
    		\draw (1,-\n) node[label=right:{\tiny\pgfmathprintnumber[fixed,precision=1]{\sumval}}] {$\circ$};}
				\foreach \n in {10,...,15}{\pgfmathsetmacro{\sumval}{2*\n+41}
    		\draw (1,-\n) node[label=right:{\tiny\pgfmathprintnumber[fixed,precision=1]{\sumval}}] {$\circ$};}
				\foreach \n in {1,...,2}{\draw (1,-\n) node[] {$\bullet$};}
				\draw (1,-12) node[] {$\bullet$};
				\draw (1,-5) node[] {$\bullet$};
				\draw (1,-4) node[] {$\bullet$};
      \end{tikzpicture}
			\caption{}
		\end{subfigure}
		\caption{$2$-abacus configuration of various $\beta$-sets.\\
			(A) $\beta$-set $Y=\set{9,7,5,3,1}$ with associated partition $P(Y) = (5,4,3,2,1)$ of size $15$.  \\
			(B) $\beta$-set $\set{65,7,5,3,1}$ corresponding to a partition of size $71$. \\
			(C) $\beta$-set $\set{65,9,5,3,1}$ corresponding to a partition of size $73$. \\
			(D) $\beta$-set $X=\set{65,9,7,2,1}$ with associated partition $P(X)=(61,6,5,2,1)$ of size $75$.}
			 \label{fig: proof example}
	\end{figure}
	\end{example}

	\begin{proof}[Proof of \cref{thm: main part A}]\label{proof: A}
		We first observe that if there exists a $\beta$-set $X$ that satisfies the following conditions: 
		\begin{enumerate}[label=(\roman*)]
			\item $P(X)\in \Part(n)$, \label{size}
			\item $\abs{X_1}\geq \abs{X_0}$ and $\abs{X_1}-\abs{X_0}=c$, and \label{core}
			\item $C_{p^k}(P(X)) > r$, \label{hooks}
		\end{enumerate}
		then $\chi^{P(X)}\in\Irr(B_c)$ is a $p$-divisible character by \cref{lem: in core c} and \cref{prop: core doesnt fit}.  

		Now, recall that \emph{either} $2 \leq c\leq 10$ and $n\geq 66$, \emph{or} $c\geq 11$ and $p^k \geq 2c-1$. To begin with, we split these two conditions into the following eight situations:
		\begin{enumerate}[label=(\alph*)]
			\item $c=2$, $n \geq 66$ and $r+3+2\delta_{a,1}+\sigma_a \leq p^k$, \label{a}
			\item $c=2$, $n \geq 66$, $r+3+2\delta_{a,1}+\sigma_a > p^k$ and $p^k \geq 10$, \label{b}
			\item $c=2$, $n \geq 66$, $r+3+2\delta_{a,1}+\sigma_a > p^k$ and $p^k < 10$, \label{c}
			\item $3\leq c \leq 10$, $n \geq 66$, $20\leq p^k$ and $n+\frac{c(c-1)}{2}-(c-1)^2 \leq ap^k-\sigma_a$, \label{d}
			\item $3\leq c \leq 10$, $n \geq 66$, $20\leq p^k$ and $n+\frac{c(c-1)}{2}-(c-1)^2 > ap^k-\sigma_a$, \label{e}
			\item $3\leq c \leq 10$, $n \geq 66$ and $20\leq p^k$, \label{f}
			\item $c\geq 11$, $2c-1\leq p^k$ and $n+\frac{c(c-1)}{2}-(c-1)^2 \leq ap^k-\sigma_a$, \label{g}
			\item $c\geq 11$, $2c-1\leq p^k$ and $n+\frac{c(c-1)}{2}-(c-1)^2 > ap^k-\sigma_a$. \label{h}
		\end{enumerate}
		If either \ref{c} or \ref{f} holds, then we show that $\Irr(B_c)$ contains a $p$-divisible character. On the other hand, if one of \ref{a}, \ref{b}, \ref{d}, \ref{e}, \ref{g} or \ref{h} holds, then we construct a $\beta$-set $X$ such that \ref{size}, \ref{core} and \ref{hooks} hold. We do so, by considering four cases, where each of the conditions \ref{a} -- \ref{h} is assumed in exactly one of the cases.

		\noindent\textbf{Case 1:} Suppose that either \ref{d} or \ref{g} holds. Let 
		\begin{align*}
			x_1 = n+\frac{c(c-1)}{2} - (c-1)^2\ \text{ and }\ x_i=2(c-i)+1\ \text{for}\ i\in[2,c].
		\end{align*}
		Since $B_c\in \Block_2(\sym_n)$ we have that $n\geq \abs{\gamma_c} = \frac{c(c+1)}{2}$ which implies that $x_1 > x_2$. Hence, $X=\set{x_1,x_2,\cdots,x_c}$ is a $\beta$-set, which satisfies \ref{size} by direct computation. It is clear that $x_i$ is odd for $i\in[2,c]$ and so \cref{lem: even is even} implies that $x_1$ is odd. Hence $X$ satisfies condition \ref{core}. 
		
		By assumption, we have that $x_1 \leq ap^k-\sigma_a$ and so $x_1 - ap^k < 0$. In order to show that $X$ satisfies condition \ref{hooks} it remains to show that $x_i<p^k$ for $i\in[2,c]$. Since $x_2 > x_3 > \dots > x_c$ it is enough to show that $x_2 \leq p^k-2$. On the one hand, if $3\leq c \leq 10$, then $x_2 = 2c-3 \leq 17 < p^k$. On the other hand, if $c\geq 11$ then $x_2 = 2c-3 < p^k$ by assumption. It follows from \cref{prop: remove hook} that \ref{hooks} holds.

		\noindent\textbf{Case 2:} Suppose that one of \ref{a}, \ref{e} or \ref{h} holds. For $i\in \NN$, let \begin{align*}
			f(n,c,p,i) := n+\frac{c(c-1)}{2} + \sigma_a + (2-a-i)p^k + 2\delta_{a,1}(i-2) -c^2 +2ci - 3i +2
		\end{align*}
		and let $s:= \min\set{i\in\NN \mid f(n,c,p,i) \leq p^k - 2(i-1+\delta_{a,1})}$. We claim that $s\leq c$. 
			\begin{itemize}
				\item If $c=2$, then 
					\begin{align*}
						f(n,2,p,2) = r + \sigma_a + 1 \leq p^k-2(1+\delta_{a,1})
					\end{align*}
					and so $s\leq c$.
				\item If $c\geq 3$ then 
					\begin{align*}
						f(n,c,p,c) &= r + (2-c)p^k + \sigma_a + 2\delta_{a,1}(c-1) +c^2-c + \frac{c(c-1)}{2}+(2-2c-2\delta_{a,1}) \\
						&\leq (3-c)p^k + \sigma_a + 2\delta_{a,1}(c-1)+c^2-c + \frac{c(c-1)}{2}-1 + (2-2c-2\delta_{a,1}).
					\end{align*}
					Now, we claim that $\sigma_a + 2\delta_{a,1}(c-1)+c^2-c + \frac{c(c-1)}{2}-1 \leq (c-2)p^k$. 
					\begin{itemize}
						\item If $3\leq c \leq 10$, then 
							\begin{align*}
								\frac{\sigma_a + 2\delta_{a,1}(c-1)+c^2-c + \frac{c(c-1)}{2} -1}{c-2} \leq \frac{3c^2+c-2}{2(c-2)} < 20 \leq p^k
							\end{align*}
							and so $\sigma_a + 2\delta_{a,1}(c-1)+c^2-c + \frac{c(c-1)}{2}-1 \leq (c-2)p^k$. 
						\item If $c\geq 11$ then 
							\begin{align*}
								\sigma_a + 2\delta_{a,1}(c-1)+c^2-c + \frac{c(c-1)}{2}-1 &\leq c^2 +c+\frac{c(c-1)}{2}-1\\ &\leq (c-2)(2c-1) \leq (c-2)p^k.
							\end{align*}
					\end{itemize}
					Hence $f(n,c,p,c) \leq p^k-2(c-1+\delta_{a,1})$ which implies that $s\leq c$.
			\end{itemize}
			Now, let
			\begin{align*}
				x_1 &= ap^k-\sigma_a,\\
				x_i &= p^k - 2(i-1+\delta_{a,1})\ \text{for}\ i\in[2,s-1],\\
				x_s &= f(n,c,p,s),\\
				x_j &= 2(c-j)+1\ \text{for}\ j\in [s+1,c].
			\end{align*}
			Next we show that $X=\set{x_1,x_2,\dots,x_c}$ is a $\beta$-set, so we need to show that $x_{s-1}>x_s$ and $x_s>x_{s+1}$. Note that the $x_{s-1}>x_2$ by the definition of $s$. Now observe that if $s=1$ then
			\begin{align*}
				f(n,c,p,1) \leq p^k -2\delta_{a,1}\ \Longleftrightarrow\ n+\frac{c(c-1)}{2} - (c-1)^2 \leq ap^k -\sigma_a
			\end{align*}
			which contradicts the assumption that $n+\frac{c(c-1)}{2} - (c-1)^2 > ap^k -\sigma_a$ in \ref{e} and \ref{h}. On the other hand, if $c=2$ then $f(n,c,p,1)\leq p^k -\delta_{a,1}$ is equivalent to $r+\sigma_a \leq 0$ which is not possible since $\sigma_a \geq 1$ and $r\geq 0$. It follows that $f(n,c,p,s-1) > p^-2(s-2+\delta_{a,1})$ by minimality of $s$. Since $f(n,c,p,s) = f(n,c,p,s-1) -p^k+2c+2\delta_{a,1}-3$, it follows that
				\[ x_s  > (p^k-2(s-2+\delta_{a,1}))-p^k+2c+2\delta_{a,1}-3 > 2(c-(s+1))+1 = x_{s+1}. \]
			Hence $X$ is a $\beta$-set. Note that $f(n,c,p,s) = n + \frac{c(c-1)}{2}-\sum_{i\neq s} x_i$ and so \ref{size} holds. It is clear that $x_i$ is odd for all $i\neq s$ by definition, and so $x_s$ is odd by \cref{lem: even is even} which implies that \ref{core} holds. Lastly, since $x_1 - ap^k \leq - \sigma_a < 0$ and $x_2 \leq p^k-2(1+\delta_{a,1})<p^k$ we have that \ref{hooks} holds by \cref{prop: remove hook}.

		\noindent\textbf{Case 3:} Suppose that \ref{b} holds. Let
			\[ x_1 = ap^k + 3,\ x_2 = r-1,\ x_3 = 3\ \text{and}\ x_4 = 1. \]
		Since $r<p^k$ we have that $x_2 \leq p^k - 2 < ap^k+3 = x_1$. Furthermore, by assumption we have that $x_2 = r-1 > p^k-3-2\delta_{a,1}-\sigma_a - 1 \geq p^k -8 \geq 3$ and so $x_2 > x_3$. Hence $X$ is a $\beta$-set. We have that \ref{size} holds by direct computation. Since $c=2$ and $\frac{c(c+1)}{2}=\abs{\gamma_c} \equiv n\ (\text{mod}\ 2)$, it follows that $n$ is odd. Hence, if $x_1$ is even then $x_2$ is odd, and vice versa, which implies that \ref{core} holds. Lastly, since $p^k\geq 10$, $x_1 - ap^k $

		\noindent\textbf{Case 4:} Suppose that either \ref{c} or \ref{f} holds. Since both \ref{c} and \ref{f} only hold for a finite number of $n$, $c$, $p$ we have by direct computation that there exists a $\lambda\in \Part(n)$ such that $\lambda\neq \lambda'$ and $\chi^\lambda \in \Irr(B_c)$ is $p$-divisible. 
	\end{proof}

	We make the following observation regarding the proof of \cref{thm: main part A}, which will be important for the proof of \cref{thm: char in block} for $G=\alt_n$ in \cref{sec:deg in alt}. Let $X$ be a $\beta$-set constructed in the proof of \cref{thm: main part A}. In Case 1 and Case 2 we note that $\ell(P(X))=c$ and so $P(X)$ is self-conjugate if and only if $x_1 = 2c-1$, which happens exactly when $n=\frac{c(c+1)}{2}$. For Case 3, $P(X)$ is never self-conjugate, since this would require $P(X)_1= 4$, and since $P(X)_1 = ap^k$ this is not possible. Hence we have that if $n$ and $c$ satisfy $n>\frac{c(c+1)}{2}$, then $\chi^{P(X)}\in \Irr(B_c)$ is $p$-divisible and $P(X)$ is not self-conjugate. Lastly, in Case 4 we have by direct computation that there exists a $\lambda\in \Part(n)$ such that $\lambda\neq \lambda'$ and $\chi^\lambda\in\Irr(B_c)$ is $p$-divisible. Hence, we have made the following observation:

	\begin{remark}\label{rmk: const of X}
		Let $p$ be an odd prime, $n\geq p$ an integer and $B_c\in \Block_2(\sym_n)$ be such that \emph{either} $2\leq c\leq 10$ and $n\geq 66$ \emph{or} $c\geq 11$ and $p^k \geq 2c-1$. Then $w(B_c)>0$ if and only if there exists some $p$-divisible $\chi^\lambda\in\Irr(B_c)$ such that $\lambda\neq\lambda'$.
	\end{remark}

	\begin{proof}[Proof of \cref{cor: N} for $G=\sym_n$]
		Since $p\geq 3$ we have that $n\geq 74$. So for blocks $B_c\in \Block_2(\sym_n)$ such that $2\leq c\leq 10$ we have that $B_c$ contains a $p$-divisible irreducible character by \cref{thm: main part A}. Hence, it remains to consider when $c\geq 11$.

		Let $\tilde{c}=\max\set{c\in \NN:\frac{c(c+1)}{2}\leq n}$. By the above, we may assume that $\tilde{c}\geq 11$. By maximality of $\tilde{c}$ we have that 
			\[ n-1 \leq \frac{(\tilde{c}+1)(\tilde{c}+2)}{2} \leq \frac{(\tilde{c}+2)^2}{2}-6 \]
		and together with the assumption that $n\geq 8p^2+2p-4$ we get $\tilde{c}+1 > 4p$. It follows that
			\[ p^k \geq \frac{n}{p} \geq \frac{\tilde{c}(\tilde{c}+1)}{2p} \geq 2\tilde{c} > 2\tilde{c}-1 \]
		Since  $2\tilde{c}-1 \leq p^k$ implies that $2c-1\leq p^k$ for all $11\leq c \leq \tilde{c}$ it follows that every $B_c\in \Block_2(\sym_n)$ contains a $p$-divisible irreducible character by \cref{thm: main part A}.
	\end{proof}

\bigskip

\section{In the case of $\alt_n$}\label{sec:deg in alt}
	The representation theory of the alternating groups is strongly connected with the representation theory of symmetric groups, for details see \cite{JK81}. In this section we will introduce the theory needed to prove \cref{thm: char in block} and \cref{cor: N} for $G=\alt_n$. Since $\alt_n$ is an index 2 subgroup of $\sym_n$, $\alt_n$ is a normal subgroup of $\sym_n$ and so it follows from Clifford theory that on the one hand, if $\lambda\neq \lambda'$ then $\chi^\lambda\res_{\alt_n} = \zeta_\lambda$ for some $\zeta_\lambda\in\Irr(\alt_n)$. On the other hand, if $\lambda=\lambda'$ then $\chi^\lambda\res_{\alt_n} = \zeta_\lambda^+ + \zeta_\lambda^-$ for some $\zeta_\lambda^+,\zeta_\lambda^-\in \Irr(\alt_n)$ and $(\zeta_\lambda^+)^g = \zeta_\lambda^-$ for any $g\in\sym_n\setminus\alt_n$. Hence for any $\zeta_\lambda\in\Irr(\alt_n)$ we have that $\zeta_\lambda(1)\in\set{\chi^\lambda(1),\frac{\chi^\lambda(1)}{2}}$. Thus if $p$ is an odd prime and $\chi\in\Irr(\sym_n)$ is $p$-divisible, then every constituent of $\chi\res_{\alt_n}$ is $p$-divisible. 

	We now proceed to describe the $2$-blocks of $\alt_n$. Let $b\in\Block_2(\alt_n)$. We say that a block $B\in\Block_2(\sym_n)$ \emph{covers} $b$ if there exists some $\chi\in \Irr(B)$ such that $b$ contains a constituent of $\chi\res_{\alt_n}$. Recall that the weight of a $2$-block $B_c$ of $\sym_n$ is $w(B_c) = \frac{\abs{\lambda}-\abs{\gamma_c}}{2}$, where $\chi^\lambda\in\Irr(B_c)$. It follows from \cite[Theorem 6.1.46]{JK81} that if $w(B_c)>0$ then $B_c$ covers a unique $2$-block $b$ of $\alt_n$, which we denote by $b_c$, and no other $2$-block of $\sym_n$ covers $b$. Hence, we denote by $b_c$ the unique $2$-block covered by $B_c\in\Block_2(\sym_n)$ with $w(B_c)>0$. In particular, we have that $\Irr(b_c) = \bigcup_{\chi\in\Irr(B_c)} \Irr(\chi\res_{\alt_n})$. On the other hand, if $w(B_c)=0$ then $\Irr(B_c)=\set{\gamma_c}$ for some $c\in\NN_0$. Since $\chi^{\gamma_c}\res_{\alt_n} = \zeta_{\gamma_c}^+ + \zeta_{\gamma_c}^-$, \cite[Theorem 6.1.46]{JK81} implies that $\zeta_{\gamma_c}^+$ and $\zeta_{\gamma_c}^-$ each lie in their own $2$-block and no other $2$-block of $\sym_n$ covers these $2$-blocks. Let $b_c^+$ and $b_c^-$ denote the $2$-blocks of $\alt_n$ such that $\Irr(b_c^+)=\set{\zeta_{\gamma_c}^+}$ and $\Irr(b_c^-)=\set{\zeta_{\gamma_c}^-}$. Hence, we have that
		\begin{align}\label{eq: block of an}
			\Block_2(\alt_n) = \set{b_c \mid B_c \in \Block_2(\sym_n),\ w(B_c) > 0} \cup \set{b_c^+, b_c^- \mid B_c\in\Block_2(\sym_c),\ w(B_c) = 0}.
		\end{align}
	The following is an immediate consequence of the above.

	\begin{proposition}\label{prop: char in alt}
		Let $p$ be an odd prime, $n\geq p$ a natural number and let $b\in\Block_2(\alt_n)$. If $B\in\Block_2(\sym_n)$ covers $b$ and $B$ contains a $p$-divisible ordinary irreducible character, then $b$ contains a $p$-divisible ordinary irreducible character. 
	\end{proposition}
	
	Before we prove \cref{thm: char in block} and \cref{cor: N} for $G=\alt_n$, we need the following lemma.

	\begin{lemma}\label{lem: irr}
		Let $n\in \NN$ and let $r\in\set{1,-1}$. Then $\prod_{i=1}^n (4i+r)$ is not a square. 
	\end{lemma}

	\begin{proof}
		For $i\in \NN$, let $x_i = 4i+r$ and let $X_n := \prod_{i=1}^n x_i$. If $n\leq 2$, then $X_n$ is not a square by direct computation. So suppose that $n\geq 3$. We proceed with a proof by contradiction. Assume that $X_n$ is a square and let $p$ denote the greatest prime divisor of $X_n$. So $p\geq 11$, and it follows from \cite[Theorem]{N52} that that there exists a prime $q$ such that $p < q < \frac{4p}{3}$, which implies that $3q < 4p$. Note that $q$ odd prime implies that either $q\equiv r\ (\text{mod}\ 4)$ or $q\equiv -r\ (\text{mod}\ 4)$. Let $s\in \NN$ be such that $x_s = q$ if $q\equiv r\ (\text{mod}\ 4)$ and if $q\equiv -r\ (\text{mod}\ 4)$ then $x_s = 3q$, so
		\begin{align}\label{find prime}
			x_s < 4p.
		\end{align}
		We now claim that $s\leq n$, i.e. $x_s\leq x_n$. Indeed, let us consider the following cases. Suppose that
		\begin{enumerate}[label=(\alph*)]
			\item $p^2 \mid x_i$ for some $i\leq n$ and that $p\equiv r\ (\text{mod}\ 4)$, then $5p \leq p^2 \leq x_i \leq x_n$.
			\item $p^2 \mid x_i$ for some $i\leq n$ and that $p\equiv -r\ (\text{mod}\ 4)$, so $7p \leq p^2 \leq x_i \leq x_n$.
			\item $p\mid x_i$ and $p\mid x_j$ for some $i<j\leq n$, and suppose that $p\equiv r\ (\text{mod}\ 4)$. The smallest distinct pair of multiples of $p$ congruent to $r$ modulo $4$ are $p$ and $5p$. It follows that $5p \leq x_j \leq x_n$.
			\item $p\mid x_i$ and $p\mid x_j$ for some $i<j\leq n$, and suppose that $p\equiv -r\ (\text{mod}\ 4)$. Then $3p$ and $7p$ are the smallest pair of distinct multiples of $p$ congruent to $r$ modulo $4$. Hence we have that $7p \leq x_j \leq x_n$.
		\end{enumerate}
		It follows from \eqref{find prime} and (a) -- (d) that $x_s<x_n$ which implies that $q\mid X_n$, contradicting maximality of $p$.
	\end{proof}

	We are now set to prove \cref{thm: char in block} and \cref{cor: N} for $G=\alt_n$. 

	\begin{proof}[Proof of \cref{thm: char in block} and \cref{cor: N} for $G=\alt_n$]
		Let $b\in\Block_2(\alt_n)$, then $b\in\set{b_c,b_c^+, b_c^-}$ for some $c\in\NN_0$ where $B_c$ is the unique $2$-block of $\sym_n$ covering $b$. It follows from \cref{thm: char in block} and \cref{cor: N} that if \emph{either} $2\leq c\leq 10$ and $n\geq 66$ \emph{or} $c\geq 11$ and $p^k\geq 2c-1$ then $\Irr(b)$ contains a $p$-divisible character. 

		We now proceed to prove the `in particular' part of \cref{thm: char in block}. Firstly, suppose that $w(b)>0$. Then $w(B_c)>0$ by definition and so there exists some $p$-divisible $\chi^\lambda\in\Irr(B_c)$ such that $\lambda$ is not self-conjugate by \cref{rmk: const of X}. Hence $\chi^\lambda\res_{\alt_n} = \zeta_\lambda \in \Irr(b)$ is $p$-divisible and since $\chi^\lambda$ is rational valued, so is $\zeta_\lambda$. Now, suppose that $w(b)=0$, then either $\Irr(b) = \set{\zeta_{\gamma_c}^+}$ or $\Irr(b) = \set{\zeta_{\gamma_c}^-}$.

 		We will proceed to show that $\zeta_{\gamma_c}^+$ and $\zeta_{\gamma_c}^-$ are not rational valued. Firstly, note that $h(\gamma_c) := ((\gamma_c)_{1,1}, (\gamma_c)_{2,2},\dots)$ is a partition of size $|\gamma_c|$ and that all parts of $h(\gamma_c)$ are odd and distinct. Let $\alpha$ denote the conjugacy class of $\sym_n$ that consists of elements with cycle type $h(\gamma_c)$. Then $\alpha$ splits into two conjugacy classes in $\alt_n$, say $\alpha^+$ and $\alpha^-$. Let $\sigma\in \alpha^+$ and $\tau\in\alpha^-$, then 
			\begin{align}\label{eq: not rational}
				\zeta_{\gamma_c}^\pm (\sigma) &= \frac{1}{2} \left( \gamma_c(\sigma) \pm \sqrt{\gamma_c(\sigma)\cdot \prod_{i=1}^{\ell(h(\gamma_c))} {h(\gamma_c)}_i} \right)\ \text{and}\\
				\zeta_{\gamma_c}^\pm (\tau) &= \frac{1}{2} \left( \gamma_c(\tau) \mp \sqrt{\gamma_c(\tau)\cdot \prod_{i=1}^{\ell(h(\gamma_c))} {h(\gamma_c)}_i} \right)
			\end{align}
		by \cite[Theorem 2.5.13]{JK81}. By the Murnaghan--Nakayama rule \cite[Theorem 2.4.7]{JK81} we have that $\gamma_c(\sigma) = \pm 1$. Furthermore, 
			\[ h(\gamma_c)_i = \begin{cases}
				4(\ell(h(\gamma_c))-i)+3 &\text{if}\ c\ \text{is even},\\
				4(\ell(h(\gamma_c))-i)+1 &\text{if}\ c\ \text{is odd},
			\end{cases} \]
		and so $\zeta_{\gamma_c}^\pm(\sigma)$ and $\zeta_{\gamma_c}^\pm(\tau)$ are not rational by \cref{lem: irr}. Hence $\zeta_{\gamma_c}^+$ and $\zeta_{\gamma_c}^-$ are not rational valued. 
	\end{proof}
\bigskip

\subsection*{Acknowledgements}
	The author would like to thank Stacey Law for their support and feedback during the work of this article. The author would also like to thank Pavel Turek for their help. 
\bigskip


\end{document}